\documentclass{amsart}
\usepackage[T1]{fontenc}
\usepackage[utf8]{inputenc}
\usepackage{amsmath,amssymb,amsthm}
\usepackage{comment}
\usepackage{hyperref}
\usepackage{tikz}

\newcommand{\RR}{\mathbb R}
\newcommand{\NN}{\mathbb N}

\renewcommand{\H}{\mathcal H}
\newcommand{\eps}{\varepsilon}

\renewcommand{\vec}[1]{\mathbf{#1}}
\newcommand{\abs}[1]{\left\vert #1 \right\vert}

\newcommand{\enclose}[1]{\left(#1\right)}
\newcommand{\Enclose}[1]{\left[#1\right]}
\newcommand{\ENCLOSE}[1]{\left\{#1\right\}}

\DeclareMathOperator{\diam}{diam}

\newcommand{\defeq}{:=}

\newtheorem{theorem}{Theorem}[section]
\newtheorem{proposition}[theorem]{Proposition}

\newtheorem{lemma}[theorem]{Lemma}

\theoremstyle{definition}

\theoremstyle{remark}
\newtheorem{remark}[theorem]{Remark}
\newtheorem{example}[theorem]{Example}

\author[Novaga]{Matteo Novaga}
\address[Matteo Novaga, Emanuele Paolini, Vincenzo Tortorelli]{Dipartimento di Matematica, Universit\`a di Pisa \\
	Largo Bruno Pontecorvo 5 \\ I-56127, Pisa}
\email[Matteo Novaga]{matteo.novaga@unipi.it}

\author[Paolini]{Emanuele Paolini} 
\email[Emanuele Paolini]{emanuele.paolini@unipi.it}

\author[Stepanov]{Eugene Stepanov}
\address[Eugene Stepanov]{%
  Scuola Normale Superiore, Piazza dei Cavalieri 6, Pisa, Italy
  \and 
  St.Petersburg Branch of the Steklov Mathematical Institute of the Russian Academy of Sciences,
	St.Petersburg, Russia
	\and
	Faculty of Mathematics, Higher School of Economics, Moscow
}
\email[Eugene Stepanov]{stepanov.eugene@gmail.com}

\author[Tortorelli]{Vincenzo Maria Tortorelli}
\email[Vincenzo Maria Tortorelli]{tortorelli@dm.unipi.it}

\thanks{
	The first and second authors are members of the INDAM/GNAMPA and were supported by the PRIN Project 2019/24.
	The work of the third author has been partially financed by the RFBR grant \#20-01-00630 A
}
\date{\today}

\title[Isoperimetric planar clusters with infinitely many regions]{Isoperimetric planar clusters with infinitely many regions}
 
\begin{document}
\maketitle

\begin{abstract}
  An infinite cluster $\vec E$ in $\RR^d$ is a sequence of disjoint 
  measurable sets $E_k\subset \RR^d$, $k\in \NN$, called regions of the cluster.
  Given the volumes $a_k\ge 0$ of the regions $E_k$, a 
  natural question is the existence of a cluster $\vec E$
  which has finite and minimal perimeter $P(\vec E)$ among all clusters with regions having such volumes.
  We prove that such a cluster exists in the planar case $d=2$, for any 
  choice of the areas $a_k$ with $\sum \sqrt a_k < \infty$.
  We also show the existence of a bounded minimizer with the property 
  $P(\vec E)=\H^1(\partial \vec E)$, where $\partial \vec E$ denotes the measure theoretic boundary of the cluster.
  %This result can be easily extended to some more general perimeter functionals.
  We also provide several examples of infinite isoperimetric clusters for anisotropic and fractional perimeters.
\end{abstract}
    
 \tableofcontents   
    
\section{Introduction}
 %Big part of research on isoperimetric problems nowadays is dedicated to so-called isoperimetric clusters. 
  A finite \emph{cluster} $\vec E$  
  %of a given finite number $N$ of  measurable sets 
  %$E_j\subset \RR^d$, $j=1,\ldots, N$ 
  %(sometimes called \textit{$N$-cluster}) 
  is a sequence
 $\vec E = (E_1, \dots E_k, \dots, E_N)$ of  measurable sets, 
 such that
 $|E_k\cap E_j|=0$ for $k\neq j$, where $|\cdot|$ denotes the Lebesgue measure (usually called \emph{volume}).
 The sets $E_j$ are called \emph{regions} of the cluster $\vec E$
 and $E_0 \defeq \RR^d\setminus\bigcup_{k=1}^{\infty} E_k$ 
 is called \emph{external region}.
 We denote the sequence of volumes of the regions of the cluster $\vec E$ as
 \begin{equation}\label{eq:def_m}
 \vec m(\vec E) \defeq %(\mathcal{L}^d(E_1), \mathcal{L}^d(E_2), \dots, \mathcal{L}^d(E_N)),
(|E_1|, |E_2|, \dots, |E_N|)
 \end{equation}
 %Supposing that for each measurable set $\subset\RR^d$ the \emph{perimeter} $P(E)\in \RR$ is finite, 
 and we call
\emph{perimeter} of the cluster the quantity
 \begin{equation}\label{eq:def_P}
 P(\vec E) \defeq \frac 1 2 \Enclose{P(E_0) + \sum_{k=1}^{N} P(E_k)},
 \end{equation}
 where $P$ is the Caccioppoli perimeter.
 %In this paper we consider the usual Caccioppoli perimeter $P$, 
 %i.e. the relaxation of the Hausdorff $(d-1)$-dimensional measure $\mathcal H^{d-1}(\partial E)$ of the topological boundary
 %$\partial E$ of $E$
 %defined on regular sets, 
A cluster $\vec E$ is called
 \emph{minimal},  or \textit{isoperimetric}, if 
 \[
 P(\vec E) = \min\ENCLOSE{P(\vec F)\colon \vec m(\vec F)=\vec m(\vec E)}.  
 \]
 %	\item[(ii)]  \emph{weakly minimal}, if 
 %	\[
 %	P(\vec E) = \inf\ENCLOSE{P(\vec F)\colon \vec m(\vec F)\ge\vec m(\vec E)}.  
 %	\]
 %	\end{itemize}
 %	We say that an $N$-cluster $\vec E$ is \emph{decomposable}
 %	if $\vec E = \vec F \cup \vec G$ (i.e.\ $E_k = F_k \cup G_k$) 
 %	with 
 %	\begin{align*}
 %	\vec m(\vec E) &= \vec m(\vec F) + \vec m(\vec G)\\
 %	P(\vec E) &= P(\vec F) + P(\vec G)\\
 %	\vec m(\vec F) &\neq \vec 0,\qquad
 %	\vec m(\vec G) \neq \vec 0
 %	\end{align*}
 
  In this paper we consider infinite clusters, i.e., 
  infinite sequences $\vec E=(E_k)_{k\in \NN}$ of 
  essentially pairwise disjoint regions: $\abs{E_j\cap E_i}=0$ for $i\neq j$
  (this can be interpreted as a model for a soap foam).
  Note that a finite cluster with $N$ regions, 
  can also be considered a particular case of an infinite cluster 
  for example by posing $E_k=\emptyset$ for $k>N$.   
  Clusters with infinitely many regions of equal areas
  were considered in~\cite{Hales01}, where it has been shown that the 
  honeycomb cluster is the unique minimizer with respect to compact perturbations.
  Infinite clusters have been considered also
  in~\cite{Leonardi02,Hirst67,Boyd82}, where 
  the variational curvature is prescribed, and in~\cite{NoPaStTo21}, 
  where existence of generalized minimizers for both finite 
  and infinite isoperimetric clusters has been proven in the general setting
  of homogeneous metric measure spaces.

\begin{figure}
  \begin{center}
    \includegraphics[width=124pt]{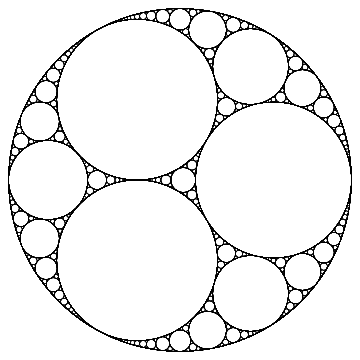}
    \hspace{1cm}
    \begin{tikzpicture}[thick]
      % python3 apollonio.py
      \draw (0,0) -- (4.0,0) -- (4.0,4.0) -- (0,4.0) -- cycle;
      \draw (2.0,0) -- (2.0,4.0); \draw(0,2.0) -- (4.0,2.0);
      \draw (1.0,0) -- (1.0,2.0); \draw(0,1.0) -- (2.0,1.0);
      \draw (0.5,0) -- (0.5,1.0); \draw(0,0.5) -- (1.0,0.5);
      \draw (0.25,0) -- (0.25,0.5); \draw(0,0.25) -- (0.5,0.25);
      \draw (0.125,0) -- (0.125,0.25); \draw(0,0.125) -- (0.25,0.125);
      \draw (0.0625,0) -- (0.0625,0.125); \draw(0,0.0625) -- (0.125,0.0625);
      \draw (0.03125,0) -- (0.03125,0.0625); \draw(0,0.03125) -- (0.0625,0.03125);
      \draw (0.015625,0) -- (0.015625,0.03125); \draw(0,0.015625) -- (0.03125,0.015625);
    \end{tikzpicture}
    \caption{The Apollonian gasket, on the left-hand side, is 
    a cluster with minimal fractional perimeter.
    On the right-hand side a similar construction with squares: this is a minimal 
    cluster with respect to the perimeter induced by the Manhattan distance.}
    \label{fig:apollonio}
\end{center}
\end{figure}
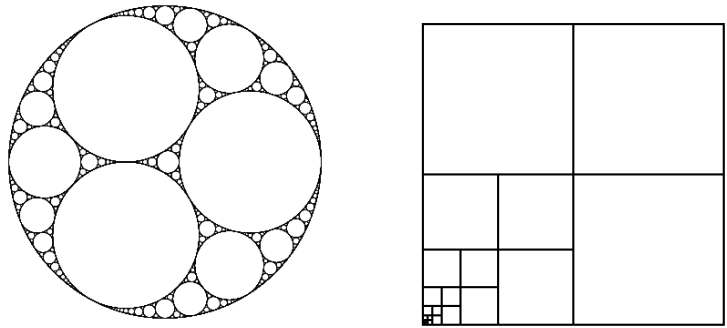
  
An interesting example of infinite cluster, detailed in Example~\ref{ex:apollonio}
(see Figure~\ref{fig:apollonio})
is the Apollonian packing of a circle (see \cite{Kasner378}).
In fact this cluster is composed by isoperimetric regions and hence 
should trivially have minimal perimeter among clusters with regions of the same 
areas. 
Actually, it turns out that this cluster has infinite perimeter 
and hence all clusters with same prescribed areas have
infinite perimeter too.
Note that very few explicit examples of minimal clusters are 
known~\cite{Foi93,Wic04,HutMorRitRos02,PaoTor20,MilNee22}.
 Nevertheless, quite curiously, Apollonian packings give nontrivial 
 examples of infinite isoperimetric clusters for fractional perimeters
 \cite{CafRoqSav10,ColMag17,CesNov20}, 
 as shown in Example~\ref{ex:apollonio}. 
 An even simpler example of an infinite isoperimetric planar 
 cluster
 is given in Example~\ref{ex:squares}
 (see Figure~\ref{fig:apollonio} again)
where the Caccioppoli perimeter is replaced by an anisotropic perimeter 
functional \cite{LawMor94,MorChrSco98,NovPao05,Car17,DePDeRGhi20}.

Our main result, Theorem~\ref{th:main}, states that 
if $d=2$ (planar case), given any sequence of positive numbers 
$\vec a = (a_1,a_2, \dots, a_k, \dots)$
such that $\sum_{k=0}^\infty \sqrt{a_k}< +\infty$,
there exists a minimal cluster $\vec E$ in $\RR^2$
with $\vec m(\vec E) = \vec a$.
The assumption on $\vec a$ is necessary to have at least 
a competitor cluster with finite perimeter.
The proof relies on two facts which are only available in the planar 
case: the isodiametric inequality for connected sets
%possibility to estimate the diameter of a connected 
%region with its perimeter,
and the semicontinuity of the length of connected sets (Go\l\c ab theorem).
%$\sum_k \sqrt{a_k} = +\infty$ 
% existence of minimal clusters is trivial
% because for any cluster $\vec E$ with $\vec m(\vec E)=\vec a$
% one has $P(\vec E) = +\infty$ by the isoperimetric inequality 
% (and this is, unfortunately, the case of Apollonian packings, 
% see Example~\ref{ex:apollonio} for details). 

% At the end of this note we provide some examples of infinite isoperimetric clusters,
% when the perimeter is replaced by an anisotropic perimeter (see~)
% or a fractional perimeter (see~). 
%We show that the Apollonian circles provide an example of a minimal  cluster for the fractional perimeter $P_s$,
%and  we discuss
%a similar example, with squares instead of circles, minimizing an anisotropic (local) perimeter $P_\phi$. 
%  Finally we provide an example of a finite cluster $\vec E=(E_1, E_2, E_3)$ 
%  such that 
%  $\H^1(\partial \vec E \setminus \partial^*\vec E)=0$ 
%  even if $\H^1(\partial E_3 - \partial^* E_3) > 0$. 

\section{Notation and preliminaries}

\subsection{Perimeters and boundaries}
For a set  $E\subset \RR^d$ with finite perimeter one can define
the \emph{reduced boundary} $\partial^* E$ which is the set 
of boundary points $x$ where the outer normal vector $\nu_E(x)$ can be defined.
One has $D1_E = \nu_E\cdot \H^{d-1}\llcorner \partial^*E$ 
where $1_E$ is the characteristic function of $E$ and 
$D1_E$ is its distributional derivative (the latter is a vector valued measure and its total variation is customarily denoted by
$|D1_E|$).
%, with outer unit normal vector denoted by $\nu_E$.  
The measure theoretic boundary of a measurable set $E$ is defined by 
\[
  \partial E 
   \defeq \{ x \in \RR^d\colon  0 < \abs{E\cap B_\rho(x)}< \abs{B_\rho(x)} \quad \mbox{for all $\rho>0$}\}.
\]
The corresponding notions for clusters can be defined as follows:
\begin{align*}
\partial^*\vec E 
&\defeq \bigcup_{k=1}^{+\infty} \bigcup_{j=0}^{k-1} \partial^* E_k \cap \partial^* E_j,\\
  \partial \vec E 
  &\defeq \big\{ x\in \RR^d\colon  
  0 < \abs{E_k\cap B_\rho(x)} < \abs{B_\rho(x)}\quad \\
  &\qquad\qquad \mbox{for all $\rho>0$ and some $k=k(\rho,x)\in\NN$}\big\}.
\end{align*}
Clearly $\partial^* \vec E\subseteq \partial \vec E$ because given $x\in \partial^*\vec E$
there exists $k$ such that $x\in \partial^* E_k$, while $\partial E_k \subseteq 
\partial \vec E$ for all $k$.
Also it is easy to check that $\partial \vec E$ is closed (and it is the closure of the 
union of all the measure theoretic boundaries $\partial E_k$). 
Moreover the following result holds true.

\begin{proposition}\label{prop:regularity}
  If $\vec E$ is a cluster with finite perimeter, then
$P(\vec E) = \H^{d-1}(\partial^*\vec E)$.
\end{proposition}
  \begin{proof}
    Consider the sets $X_n$, for $1\le n \le \infty$, defined by
    \[
      X_n \defeq \ENCLOSE{x\in \RR^d\colon \#\ENCLOSE{k\in \NN \colon x\in \partial^* E_k}= n}
    \]
    (notice that $k=0\in \NN$, the external region, is included in the count).   
 It is clear that $X_n=\emptyset$ for all $n\ge 3$ because in every point of $\partial^* E_k$ there is an approximate tangent hyper-plane which can only be shared by two regions. 

We claim that $\H^{d-1}(X_1)=0$. To this aim suppose by contradiction that $\H^{d-1}(X_1)>0$. Then there exists a $j\in \NN$ such 
    that 
    \[\abs{D1_{E_j}}(X_1)= \H^{d-1}(X_1\cap \partial^* E_j)>0,\] because $X_1$ is contained in the countable union
    $\cup_{j=0}^\infty X_1\cap \partial^* E_j$.
    Hence there is a subset $A\subset X_1\cap \partial^* E_j$ such that $D 1_{E_j}(A)\neq \vec 0$.
    Notice that $\sum_{k=0}^\infty 1_{E_k} = 1$, %($k=0,1,\dots$) 
    hence also
    $\sum_k D 1_{E_k} = \vec 0$.
    Since $D 1_{E_j}(A)\neq \vec 0$ there must exist at least another index $k\neq j$ 
    such that $D 1_{E_k}(A)\neq \vec 0$, hence  
    $\H^{d-1}(A\cap \partial^*E_k)>0$. But then 
    \[\emptyset\not=
    A\cap \partial^*E_k\subset X_1\cap \partial^* E_j\cap \partial^*E_k, \quad j\neq k, \] contrary to the definition of $X_1$, which proves the claim. 
    
    In conclusion, the union of all reduced boundaries $\partial^* E_k$ is contained in $X_2$ 
    up to a $\H^{d-1}$-negligible set. Hence
        \begin{align*}
P(\vec E) 
      &= \frac 1 2 \sum_{k=0}^{+\infty} P(E_k) 
      = \frac 1 2 \sum_{k=0}^{+\infty} \H^{d-1}(\partial^*E_k \cap X_2)=\\
      &=\frac 1 2 \sum_{k=0}^{+\infty}\sum_{j\neq k} \H^{d-1}(\partial^*E_k\cap \partial^* E_j)= \sum_{k=0}^{+\infty}\sum_{j=k+1}^{+\infty} \H^{d-1}(\partial^*E_k\cap \partial^* E_j)=\\
    &  = \H^{d-1}\left(\bigsqcup_{k=0}^{+\infty}\bigsqcup_{j=k+1}^{+\infty}\partial^*E_k\cap \partial^*E_j\right)= \H^{d-1}(\partial^* \vec E)
    \end{align*}
    as claimed.
  \end{proof}

\subsection{Auxiliary results}

In the following theorem we collect known existence and regularity results 
for \textit{finite} minimal clusters from~\cite{Morgan87,Maggi12-GMTbook}.

\begin{theorem}[existence and regularity of planar $N$-clusters]
  \label{th:minimal}
  Let $a_1, a_2,$ $ \dots, a_N$ be given positive real numbers.
  Then there exists a minmal $N$-cluster $\vec E = (E_1,\dots E_N)$
  in $\RR^s$,
  with $\abs{E_k} = a_k$ for $k=1,\dots, N$.
  If $\vec E$ is a minimal $N$-cluster and $d=2$, then
  $\partial \vec E$ is a pathwise connected set composed by circular arcs or line segments
  joining in triples at a finite number of vertices. 
  Moreover in this case $P(\vec E) = \H^1(\partial \vec E)$.
\end{theorem}

The statement below gives isodiametric inequality for planar finite clusters,

\begin{proposition}[diameter estimate]\label{prop:bounded}
If $\vec E$ is an $N$-cluster in $\RR^2$ and $\partial \vec E$ 
is pathwise connected, then
\[
  \diam \partial \vec E \le P(\vec E).
\]
\end{proposition}
\begin{proof}
Since $\partial \vec E$ is pathwise connected, given any two points $x,y\in \partial \vec E$ 
we find that $\abs{x-y} \le \H^1(\partial \vec E) = P(\vec E)$.
\end{proof}
  
  Another ingredient will be the following statement on cluster truncation,
  
\begin{proposition}[cluster truncation]\label{prop:trunk} 
  Let $\vec E = (E_1,\dots, E_k,\dots)$ be a (finite or infinite) cluster
  and let $T_N \vec E$ be the $N$-cluster $(E_1, \dots, E_N)$.
  %If $\sum_{k=1}^{\infty}\abs{E_k}<+\infty$  
  Then
  \[
  P(T_N \vec E) \le P(\vec E).
  \]
\end{proposition}

\begin{proof}
  For measurable sets $E,F$ the inequality
\[    P(E\cup F) + P(E\cap F) \le P(E) + P(F)
\] 
holds, 
  hence if $\vert E\cap F\vert=0$, one has
  \[
    P(E) 
    = P((E\cup F)\cap (\RR^d\setminus F)) 
    \le P(E\cup F) + P(F).
  \]
  %So we have 
  %\[
  %  P\enclose{\bigcup_{k=1}^n E_k} 
   % \le P\enclose{\bigcup_{k=1}^{\infty} E_k} + P\enclose{\bigcup_{k=n+1}^{\infty} E_k}.
  %\]
  It follows that
  \begin{align*}
    2 P(T_N\vec E) 
    &= \sum_{k=1}^n P(E_k) + P\enclose{\bigcup_{k=1}^n E_k} \\
    &\le \sum_{k=1}^n P(E_k) + P\enclose{\bigcup_{k=1}^\infty E_k} + P\enclose{\bigcup_{k=n+1}^\infty E_k}\\
    &\le  \sum_{k=1}^n P(E_k) + P\enclose{\bigcup_{k=1}^\infty E_k} + \sum_{k=n+1}^\infty P(E_k)\\
    &= \sum_{k=1}^\infty P(E_k) + P\enclose{\bigcup_{k=1}^\infty E_k} = 2P(\vec E)
  \end{align*}
  as claimed.
\end{proof}

\begin{lemma}\label{lm:boundary}
  Let $E$ be a measurable set and $\Omega$ an open connected set.
  If $\partial E\cap \Omega = \emptyset$, then either
  $\abs{\Omega \cap E} = 0$ or $\abs{\Omega\setminus E}=0$.
\end{lemma}

\begin{proof}
Notice that $\displaystyle 
  \Omega\setminus \partial E = A_0 \cup A_1$, 
where 
\begin{align*}
  A_0 &\defeq \ENCLOSE{x\in \Omega \colon \abs{B_\rho(x)\cap E}=0\quad\mbox{for some $\rho>0$}},\\  
  A_1 &\defeq \ENCLOSE{x\in \Omega \colon \abs{B_\rho(x)\setminus E}=0\quad\mbox{for some $\rho>0$}}.
\end{align*}
It is clear that $A_0$ and $A_1$ are open disjoint sets, and 
if  $\partial E\cap \Omega = \emptyset$ their union is the whole set $\Omega$.
If $\Omega$ is connected, then either $A_0$ or $A_1$ is equal to $\Omega$ 
which means that either $\abs{\Omega\cap E}=0$ or $\abs{\Omega\setminus E}=0$.
\end{proof}

%The following theorem can be found in \cite[Theorem 4.4.7 page 67]{AmbTil04}.
%\begin{theorem}[rectifiability]\label{th:rectifiable}
%If $\Gamma\subset \RR^d$ is a compact connected set 
%with $\H^1(\Gamma)<+\infty$, then $\Gamma$ is rectifiable.
%\end{theorem}
%
%We also recall the following classical result 
%(see \cite[Theorem 4.4.17]{AmbTil04}, \cite{PaoSte13}).
%
%\begin{theorem}[Go\l\c ab]\label{th:Golab}
%  Let $K_n$ be a sequence of compact and connected subsets of $\RR^d$
%  which converges to a compact set $K\subset \RR^d$ with respect to 
%  the Hausdorff distance.
%  Then $K$ is connected and 
%  \[
%  \H^1(K) \le \liminf_k \H^1(K_n).  
%  \]
%\end{theorem}

\section{Main result}

The statement below provides existence of infinite planar isoperimetric clusters.

\begin{theorem}[existence]%
\label{th:main}%
Let $\vec a = (a_1, \dots, a_k, \dots)$ be a sequence 
of nonnegative numbers such that $\sum_{k=1}^\infty \sqrt{a_k}<\infty$. 
Then there exists a minimal cluster $\vec E$ in $\RR^2$, 
with $\vec m(\vec E)=\vec a$ satisfying additionally
\begin{gather}\label{eq:existence1}
  \bigcup_{k=1}^{\infty} E_k \text{ is bounded,}\\
\label{eq:connected}
 \partial \vec E \text{ is pathwise connected,}\\
 \label{eq:existence2}
 \H^1(\partial \vec E \setminus \partial^*\vec E) = 0.
\end{gather} 
\end{theorem}
\begin{remark}
In view of~\eqref{eq:existence2} and Proposition~\ref{prop:regularity},
for the minimal cluster provided by the above
Theorem~\ref{th:main},
one has 
\begin{equation}\label{eq:PH1}
  P(\vec E) = \H^1(\partial \vec E) = \H^1(\partial^* \vec E).
\end{equation}
Of course there exists a set with finite perimeter $E$ such that 
$P(E)<\H^1(\partial E)$ hence~\eqref{eq:PH1} is false for general, 
non minimal, clusters.

It is interesting to note that,
as shown in example~\ref{ex:circles},
there exists a finite cluster $\vec E$ satisfying~\eqref{eq:PH1},
for which one does not have $P(E_k)=\H^1(\partial E_k)$ 
for all $k$.
It would be interesting to see whether 
these equalities hold for minimal clusters.
\end{remark}

\begin{proof} Let $\displaystyle
    \bar p \defeq 2\sqrt{\pi}\sum_{k=1}^{\infty} \sqrt{a_k} < +\infty$, 
 and
\begin{align*}
  p &\defeq \inf \{P(\vec E)\colon \text{$\vec E$ cluster in $\RR^2$ with $\lvert E_k\rvert=a_k$, $k=1,2,\dots, n, \dots$}\},\\
  p_n &\defeq \inf \{P(\vec E)\colon \text{$\vec E$ $n$-cluster in $\RR^2$ with $\lvert E_k\rvert=a_k$, $k=1,\dots, n$}\}
\end{align*}
so that a cluster $\vec E$ with measures $\vec m(\vec E) = \vec a$ 
is minimal if and only if $P(\vec E) = p$,
while an $n$-cluster $\vec E$ with measures $\abs{E_k}=a_k$ for $k=1,\dots,n$ 
is minimal if and only if $P(\vec E) = p_n$.

If $\vec E$ is a competitor for $p$, then $T_n \vec E$ is a competitor for $p_n$
and, by Proposition~\ref{prop:trunk}, one has $P(T_n \vec E) \le P(\vec E)$. 
Hence $p_n \le p$.
Moreover one can build a competitor for $p$ which is composed 
by circular disjoint regions $(B_1,\dots, B_j,\dots )$,
where $B_j$ are disjoint balls of radii $\sqrt{\frac{a_j}{\pi}}$,
to find that $p \le \bar p < +\infty$.

For each $n\ge 1$ consider a minimal $n$-cluster $\vec F^n$
with $\lvert F^n_k\rvert = a_k$ for $k\le n$, $F^n_k\defeq\emptyset$ for $k>n$
so that $P(\vec F^n) = p_n$.
Hence, by Proposition~\ref{prop:bounded}, 
up to translations we might and shall suppose that all 
the regions $F^n_k$ of all the  
clusters $\vec F^n$ are contained in a ball of radius $\bar p$.
In fact:
\[
  \bar p
  \ge p 
  \ge \sup_n p_n 
  = \sup_n P(\vec F^n) 
  \ge \sup_n \diam \partial \vec F^n.
\]

Up to a subsequence we can hence assume that the first regions $F^n_1$
converge to a set $E_1$ in the sense that their characteristic
functions $\mathbf 1_{F_1^n}$ converge to the characteristic 
function $\mathbf 1_{E_1}$ in the Lebesgue space $L^1(\RR^2)$
(we call this convergence $L^1$ convergence of sets).
Analogously, up to a sub-subsequence 
also the second regions $F^n_2$ 
converge in $L^1$ sense to a set $E_2$, 
and in this way we define inductively
the sets $E_k$ for all $k\ge 1$.
Then there exists a diagonal subsequence with indices $n_j$ such that for all $k$
one has $F^{n_j}_k \to E_k$ in $L^1$ for all $k\ge 1$ as $j\to +\infty$.

Consider the cluster $\vec E$ with the components $E_k$ 
defined above.
By continuity we have $\vec m(\vec E) = \vec a$ because 
$F^{n_j}_k \to E_k$ in $L^1$ as $j\to \infty$ and
$\abs{F^{n_j}_k} = a_k$ for all $j$.  
We claim that the union of all the regions of $\vec F^{n_j}$ also converges to the 
union of all the regions of $\vec E$.
For all $\eps>0$ take $N$ such that 
$\sum_{k=N+1}^\infty a_k \le \eps$
and notice that 
\[  
\enclose{\bigcup_{k=1}^{\infty} E_k} \triangle \enclose{\bigcup_{k=1}^{\infty} F_k^{n_j}}
\subseteq \bigcup_{k=1}^N \enclose{E_k \triangle F_k^{n_j}}
\cup \bigcup_{k=N+1}^{\infty} E_k \cup \bigcup_{k=N+1}^{\infty} F_k^{n_j}.
\]
Hence 
\[
  \limsup_j \abs{\bigcup_{k=1}^{\infty} E_k \triangle \bigcup_{k=1}^{\infty} F_k^{n_j}}
  \le \lim_j \sum_{k=1}^N \abs{E_k \triangle F_k^{n_j}} + 2 \eps = 2\eps.
\]
Letting $\eps \to 0$ we obtain the claim.

By lower semicontinuity of perimeter:
\[  
  P(E_k)  \le \liminf_{j\to +\infty} P(F^{n_j}_k)
  \qquad\text{and}\qquad
  P\left(\bigcup_{k=1}^{+\infty} E_k\right) 
  \le \liminf_{j\to+\infty} P\left(\bigcup_{k=1}^{+\infty} F^{n_j}_k\right)
\]
and hence $P(\vec E) \le \liminf_j P(\vec F^{n_j}) \le p$ proving 
that $\vec E$ is actually a minimal cluster. 
Since all the regions $F_k^n$ are equi-bounded we obtain~\eqref{eq:existence1}.

We are going to prove~\eqref{eq:existence2}.
By Theorem~\ref{th:minimal} the minimal $n$-cluster $\vec F^n$
has a measure theoretic boundary $\partial \vec F^n$ which is a compact 
and connected set such that $P(\vec F^n) = \H^1(\partial \vec F^n)$.
Up to a subsequence, the compact sets $\partial \vec F^{n_j}$, 
being uniformly bounded, converge with respect to the Hausdorff distance,
to a compact set $K$. 
Without loss of generality suppose $n_j$ is labeling this new subsequence.

We claim that $\partial \vec E \subseteq K$.
In fact for any given $x\in \partial \vec E$
and any $\rho>0$ there exists $k=k(\rho)$ such that 
$B_\rho(x)\cap E_k$ and $B_\rho(x)\setminus E_k$ 
both have positive measure. 
Since $\abs{B_\rho(x)\cap F^{n_j}_k}\to \abs{B_\rho(x)\cap E_k} >0$
and $\abs{B_\rho(x)\setminus F^{n_j}_k} \to \abs{B_\rho(x)\setminus E_k} >0$
for $j=j(\rho)$ sufficiently large 
by Lemma~\ref{lm:boundary}
there is a point $x_k^j \in B_\rho(x)\cap \partial F^{n_j}_k$.
As $\rho\to 0$ the sequence $x_k^{j}$ converges to $x$ and 
since $\partial F_k^{n_j}\subseteq \partial \vec F^{n_j}$ we conclude
that $x\in K$.

The sets $\partial \vec F^n$ are connected, hence,
by the classical Go\l\c ab theorem on semicontinuity of one-dimensional Hausdorff measure over sequences of connected sets 
(see~\cite[theorem 4.4.17]{AmbTil04} or~\cite[theorem 3.3]{PaoSte13} for its most general statement and a complete proof),
one has
%Theorem~\ref{th:Golab}, one has 
\[
  \H^1(K)
  \le \liminf_n \H^1(\partial \vec F^n)
\]
and $K$ is itself connected.
Summing up and using Proposition~\ref{prop:regularity}
\begin{equation}\label{eq:5867}
  \begin{aligned}
P(\vec E) 
&= \H^1(\partial^* \vec E)  
\le \H^1(\partial \vec E) 
\le \H^1(K) \\
&\le \liminf_n P(\vec F^n)
\le \limsup_n p_n \le p\le P(\vec E)
  \end{aligned}
\end{equation}
hence $\H^1(\partial^* \vec E) = \H^1(\partial \vec E) = \H^1(K)$, $p_n\to p$ 
and \eqref{eq:existence2} follows.

Finally, to prove that $\partial \vec E$ is connected, it is 
enough to show $\partial \vec E = K$.
We already know that $\partial \vec E\subseteq K$ so 
we suppose by contradiction that there exists $x\in K\setminus \partial \vec E$.
Take any $y\in K$. 
The set $K$ is arcwise connected by rectifiable arcs, 
since it is a compact connected set of finite one-dimensional Hausdorff measure
(see e.g.~\cite[lemma 3.11]{Fal86} or \cite[theorem 4.4.7]{AmbTil04}), 
in other words,
%Theorem~\ref{th:rectifiable} 
there exists an injective continuous curve 
$\gamma\colon[0,1]\to K$ with $\gamma(0) = x$ and $\gamma(1)=y$. 
Since $\partial \vec E$ is closed in $K$
there is a small $\eps>0$ such that 
$\gamma([0,\eps])\subset K\setminus \partial \vec E$ 
and hence $\H^1(K\setminus \partial \vec E)>0$
contrary to $\H^1(K) = \H^1(\partial \vec E)$; this contractictions shows the last claim and hence concludes the proof.
\end{proof}

% \begin{remark}
%   As easily observed from the above proof, the statement 
%   of the Theorem~\ref{th:main} remains true when instead of Euclidean 
%   (Caccioppoli) perimeter one has a more general perimeter functional P
% \end{remark}

\section{Some examples}

We collect here some interesting examples 
of infinite planar clusters.

\begin{example}[Apollonian packing] 
\label{ex:apollonio}%
A cluster $\vec E$, as depicted in Figure~\ref{fig:apollonio},
can be constructed so that each region $E_k = B_{r_k}(x_k)$, $k\neq 0$, 
is a ball contained in the ball $B_1 = \RR^2\setminus E_0$. 
The balls can be choosen to be pairwise disjoint and such that
the measure of $B_1 \setminus \bigcup_{k=1}^{\infty} E_k = 0$
(see~\cite{Kasner378}). 
%% Kasner e Supnick dimostrano che l'insieme residuale ha misura nulla

Clearly such a cluster must be minimal because each region $E_k$ 
has the minimal possible perimeter among sets with the given area
and the same is true for the complement of the exterior region 
$E_0$ which is their union.
The boundary $\partial \vec E$ of such a cluster 
is the \emph{residual set}, 
i.e.\ the set of zero measure which remains when the balls $E_k$ 
are removed from the large ball $\overline{B_1}$:
\begin{equation}
  \partial \vec E = \bigcup_{k=0}^{+\infty} \partial B_{r_k}(x_k)
   = \overline B_1 \setminus \bigcup_{k=1}^{+\infty} B_{r_k}(x_k).
\end{equation}
Unfortunately the residual set of such a cluster
has Hausdorff dimension $d>1$ (see~\cite{Hirst67}) 
and hence the cluster $\vec E$ cannot have finite perimeter.

% Hirst dimostra che l'insieme residuale (che è noto abbia area nulla) ha dimensione 
% di hausdorff inferiore a 1.44 ed è strettamente maggiore di 1.

However we can consider
the fractional (non local) perimeter $P_s$ defined by 
\[
  P_s(E) = \int_E \int_{\RR^2\setminus E} \frac 1 {\abs{x-y}^{2+s}} \, dx\, dy  
\]
to define the corresponding non local perimeter
$P_s(\vec E)$ of the cluster $\vec E$ by means of definition~\eqref{eq:def_P}
with $P_s$ in place of $P$.
If $r_k$ is the radius of the $k$-th disk of the cluster 
it turns out (see \cite{Boyd82}) that the infimum of all $\alpha$, such that 
the series $\sum_k r_k^\alpha$ converges,
is equal to $d$, the Hausdorff dimension of $\partial \vec E$.
Since $d \in (1,2)$ for all $s < 2-d$ we have  
\[
  \sum_k r_k^{2-s} < +\infty  
\]
and since $P_s(B_r) = C\cdot r^{2-s}$ (with $0<C<+\infty$)
we obtain $P_s(\vec E)<+\infty$ for such $s$.
It is well known (see \cite{FrankSeiringer10}) that the solution 
to the fractional isoperimetric problem is given by balls, 
hence $\vec E$ provides an example of an infinite minimal cluster 
with respect to the fractional perimeter $P_s$.
\end{example}

\begin{example}[Anisotropic isoperimetric packing] 
\label{ex:squares}%	
	We can find 
a similar example if we consider 
an anisotropic perimeter such that the isoperimetric problem has 
the square (instead of the circle) as a solution.
If $\phi$ is any norm on $\RR^2$ one can define the perimeter $P_\phi$ 
which is the relaxation of the following functional defined 
on regular sets $E\subset \RR^2$:
\[
  P_\phi(E) = \int_{\partial E} \phi(\nu_E(x))\, d \H^{1}(x)
\]
where $\nu_E(x)$ is the exterior unit normal vector to $\partial E$ 
in $x$. 
If $\phi(x,y) = \abs{x} + \abs{y}$ (the \emph{Manhattan} norm) it is well known that the $P_\phi$-minimal 
set with prescribed area (i.e.\ the \emph{Wulff shape}) 
is a square with sides parallel to the coordinated axes (which is the ball for the dual norm).
It is then easy to construct an infinite cluster $\vec E=(E_1,\dots, E_k\dots)$ 
where each $E_k$ is a square and also the union of all such squares is a square, 
see figure~\ref{fig:apollonio}.
By iterating such a construction it is not difficult to realize that 
given any sequence $a_k$, $k=1,\dots, n,\dots$ of numbers such that 
their sum is equal to $1$ and each number is a power of $\frac 1 4$ it is 
possible to find a cluster $\vec E$ with $\vec m(\vec E) = \vec a$ such 
that each $E_k$ is a square and the union $\bigcup_k E_k$ is the unit square.
\end{example}

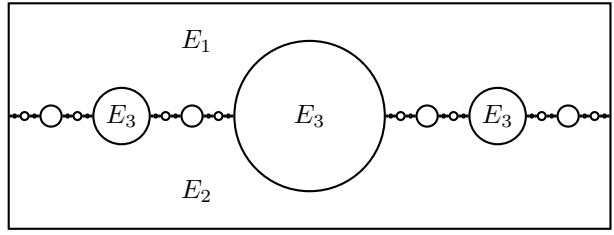
\begin{figure}
  \begin{center}
    \begin{tikzpicture}[thick]
    \node at (2.5,1) {$E_1$};
    \node at (2.5,-1) {$E_2$};
    \node at (4,0) {$E_3$};
    \node at (1.5,0) {$E_3$};
    \node at (6.5,0) {$E_3$};
    % python3 bad.py
    \draw (8,1.5) -- (0,1.5) -- (0,-1.5) -- (8,-1.5) -- cycle; 
    \draw (4.0,0) circle (1.0);
    \draw (1.5,0) circle (0.375);
    \draw (0.5625,0) circle (0.140625);
    \draw (0.2109375,0) circle (0.052734375);
    \draw (0.0791015625,0) circle (0.019775390625);
    \draw (0.0296630859375,0) circle (0.007415771484375);
    \draw (0.1285400390625,0) circle (0.007415771484375);
    \draw (0.3427734375,0) circle (0.019775390625);
    \draw (0.2933349609375,0) circle (0.007415771484375);
    \draw (0.3922119140625,0) circle (0.007415771484375);
    \draw (0.9140625,0) circle (0.052734375);
    \draw (0.7822265625,0) circle (0.019775390625);
    \draw (0.7327880859375,0) circle (0.007415771484375);
    \draw (0.8316650390625,0) circle (0.007415771484375);
    \draw (1.0458984375,0) circle (0.019775390625);
    \draw (0.9964599609375,0) circle (0.007415771484375);
    \draw (1.0953369140625,0) circle (0.007415771484375);
    \draw (2.4375,0) circle (0.140625);
    \draw (2.0859375,0) circle (0.052734375);
    \draw (1.9541015625,0) circle (0.019775390625);
    \draw (1.9046630859375,0) circle (0.007415771484375);
    \draw (2.0035400390625,0) circle (0.007415771484375);
    \draw (2.2177734375,0) circle (0.019775390625);
    \draw (2.1683349609375,0) circle (0.007415771484375);
    \draw (2.2672119140625,0) circle (0.007415771484375);
    \draw (2.7890625,0) circle (0.052734375);
    \draw (2.6572265625,0) circle (0.019775390625);
    \draw (2.6077880859375,0) circle (0.007415771484375);
    \draw (2.7066650390625,0) circle (0.007415771484375);
    \draw (2.9208984375,0) circle (0.019775390625);
    \draw (2.8714599609375,0) circle (0.007415771484375);
    \draw (2.9703369140625,0) circle (0.007415771484375);
    \draw (6.5,0) circle (0.375);
    \draw (5.5625,0) circle (0.140625);
    \draw (5.2109375,0) circle (0.052734375);
    \draw (5.0791015625,0) circle (0.019775390625);
    \draw (5.0296630859375,0) circle (0.007415771484375);
    \draw (5.1285400390625,0) circle (0.007415771484375);
    \draw (5.3427734375,0) circle (0.019775390625);
    \draw (5.2933349609375,0) circle (0.007415771484375);
    \draw (5.3922119140625,0) circle (0.007415771484375);
    \draw (5.9140625,0) circle (0.052734375);
    \draw (5.7822265625,0) circle (0.019775390625);
    \draw (5.7327880859375,0) circle (0.007415771484375);
    \draw (5.8316650390625,0) circle (0.007415771484375);
    \draw (6.0458984375,0) circle (0.019775390625);
    \draw (5.9964599609375,0) circle (0.007415771484375);
    \draw (6.0953369140625,0) circle (0.007415771484375);
    \draw (7.4375,0) circle (0.140625);
    \draw (7.0859375,0) circle (0.052734375);
    \draw (6.9541015625,0) circle (0.019775390625);
    \draw (6.9046630859375,0) circle (0.007415771484375);
    \draw (7.0035400390625,0) circle (0.007415771484375);
    \draw (7.2177734375,0) circle (0.019775390625);
    \draw (7.1683349609375,0) circle (0.007415771484375);
    \draw (7.2672119140625,0) circle (0.007415771484375);
    \draw (7.7890625,0) circle (0.052734375);
    \draw (7.6572265625,0) circle (0.019775390625);
    \draw (7.6077880859375,0) circle (0.007415771484375);
    \draw (7.7066650390625,0) circle (0.007415771484375);
    \draw (7.9208984375,0) circle (0.019775390625);
    \draw (7.8714599609375,0) circle (0.007415771484375);
    \draw (7.9703369140625,0) circle (0.007415771484375);
\end{tikzpicture}
    \caption{An example of a cluster $\vec E$ with finite perimeter 
    such that $P(\vec E)=\H^1(\partial \vec E)$ 
    but $P(E_3)< \H^1(\partial E_3)$.
    }
    \label{fig:bad}
\end{center}
\end{figure}

\begin{example}[Cantor circles]
  \label{ex:circles}%
  See Figure~\ref{fig:bad} and \cite[example 2 pag. 59]{ACMM01}.
  Take a rectangle $R$ divided in two by segmente $S$ 
  on its axis.
  Let $C$ be a Cantor set with positive measure 
  constructed on $S$.
  Consider the set $E_3$ which is the union 
  of the balls with diameter on the intervals composing 
  the complementary set $S\setminus C$.
  Let $E_1$ and $E_2$ be the two connected
  components of $R\setminus \overline{E_3}$. 
  It turns out that the $3$-cluster $\vec E = (E_1, E_2, E_3)$ has finite perimeter 
  and the perimeter of $\vec E$ is represented by the Hausdorff 
  measure of the boundary:
  \[
    P(\vec E) 
    = \H^1(\partial \vec E).
  \]
  However the same is not true for each region. 
  In fact the boundary $\partial E_3$ of the region $E_3$ includes $C$ and hence 
  \[
    P(E_3)<   
    \H^1(\partial E_3).
  \]
\end{example}

\bibliographystyle{plain}
%\bibliography{isoperim}

\end{document}